\documentclass{amsart}
\usepackage{amsmath,amssymb,latexsym}[24 April 2018, rev. 8 October 2018, rev. 26 July 2019]
\usepackage{amscd}
\newtheorem{proposition}{Proposition}[section]

\newtheorem{lemma}[proposition]{Lemma}
\newtheorem{corollary}[proposition]{Corollary}
\newtheorem{theorem}[proposition]{Theorem}
\newtheorem{remark}[proposition]{Remark}

\newtheorem{problem}[proposition]{Problem}

\newcommand{\dpt}{\operatorname{depth}}
\newcommand{\rank}{\operatorname{rank}}
\newcommand{\supp}{\operatorname{supp}}
\newcommand{\im}{\operatorname{Im}}
\newcommand{\Ker}{\operatorname{Ker}}
\newcommand{\Id}{\operatorname{Id}}
\newcommand{\Sing}{\operatorname{Sing}}
\newcommand{\La}{\Lambda}
\newcommand{\Om}{\Omega}

\def \R     {\Bbb R}
\def \C     {\Bbb C}
\def \N     {\Bbb N}
\def \FF    {{\mathcal F}}
\def \GG    {{\mathcal G}}
\def \HH    {{\mathcal H}}
\def \MM    {{\mathcal M}}
\def \RR    {{\mathcal R}}
\def \SS    {{\mathcal S}}
\def \mm    {{ m}}
\def \ss    {{ s}}
\def \OO    {{\mathcal O}}

\def \o     {\omega}
\def \g     {\gamma}
\def \w     {\wedge}

\def \nzd  {non-zero-divisor\ }
\def \bR {\bar R}

\def \tRi {\tilde R_i}
\def \tRni {\tilde R_{n,i}}

\def \Ra  {R_{[a]}}
\def \MMa {\MM_{[a]}}
\def \bMM {\bar\MM}
\def \ba  {\bar a}
\def \bo  {\bar \omega}
\def \bg  {\bar \gamma}
\def \bO  {\bar \Omega}
\def \bxi  {\bar \xi}

\begin{document}

\title{
Exterior multiplication with singularities:\\ a Saito's theorem in vector bundles\\
\footnotesize{\rm In memory of Pawe{\l} Doma\'nski}}

\author{B. Jakubczyk}

\address{\noindent Institute of Mathematics, Polish Academy of Sciences,  \'Sniadeckich 8, 00-956 Warsaw, Poland}

\email{b.jakubczyk@impan.pl}


\begin{abstract}
Let $E$ be a vector bundle over a suitable differentiable manifold $M$ and let $\w^pE$ denote $p$-exterior product of $E$.  Given sections $\o_1,\dots,\o_k$ of $E$ and a section $\eta$ of $\w^pE$, we consider the problem if $\eta$ can be written in the form
\[
\eta=\sum \o_i\w\g_i,
\]
where $\g_i$ are sections of $\w^{p-1}E$. An obvious necessary condition $\Om\w\eta=0$, where $\Om=\o_1\w\cdots\w\o_k$, has to be supplemented with a condition that the form $\Om$ has sufficiently regular singularities at points where $\Om(x)=0$. Such a local condition is suggested by an algebraic theorem of K. Saito and is given in terms of the depth of the ideal defined by coefficients of $\Om$. Working in the smooth, real analytic or holomorphic (with $M$ Stein manifold) category, we show that the condition is sufficient for the above property to hold. Moreover, in the smooth category it is sufficient for existence of a continuous right inverse to the operator defined by $(\g_1,\dots,\g_k)\mapsto\sum \o_i\w\g_i$. All these results are also proven in the case where $E$ is a bundle over a suitable closed subset of $M$.
\end{abstract}

\maketitle

\section{Introduction.}

Consider the following problems. Given differential 1-forms $\o_1,\dots,\o_k$ and a differential $p$-form $\eta$ on a manifold $M$, we ask if there exist  $(p-1)$-forms $\g_1,\dots,\g_k$ such that $\eta=\sum\o_i\w\g_i$? If the answer is positive and the forms $\o_1,\dots,\o_k$ are fixed, can we find a linear continuous (in the $C\sp\infty$ topology) operator $\eta\mapsto(\g_1,\dots,\g_k)$ which solves the problem?

To examine the first of the above problems note that the obvious necessary condition on $\eta$ is $\Om\wedge\eta=0$, where $\Om=\o_1\w\cdots\w\o_k$. It is also sufficient, locally around points $x\in M$ where $\Om(x)\not=0$, and the difficulty of the first problem is reduced to points where $\Om$ vanishes. Both questions appear, e.g., when the Moser homotopy method is used in problems of equivalence of contact or symplectic structures with singularities (see \cite{JZ1}, \cite{JZ2}). In order to apply the method one has to solve the so called homological equation and is faced with both above problems (\cite{JZ2}).

In this paper we consider the first problem for $C^r$ sections $\o_1,\dots,\o_k$ of a vector bundle $E$ over $M$, $C^r$ sections $\eta$ of the exterior product bundle $\w^pE$ and $C^r$ sections $\g_1,\dots,\g_k$ of $\w^{p-1}E$. Here $C^r$ is one of the categories $C^\infty$, $C^\omega$ (real analytic), or $C^{hol}$ (holomorphic). We restrict the class of possible singularities of $\Om$ by introducing a depth condition on the ideals of germs generated by the coefficients of $\Om$. Using an extension to non-Noetherian rings of an algebraic theorem of K. Saito \cite{Sa} (Theorem \ref{thmSaito}) we show that under the depth condition \eqref{depth1} the necessary condition $\Om\wedge\eta=0$ is also sufficient for our property to hold (Theorem \ref{thm1}). Moreover, in the $C^\infty$ category we show that the same condition is sufficient for the existence of a continuous linear operator $\eta\mapsto (\g_1,\dots,\g_k)$ which is right inverse to $(\g_1,\dots,\g_k)\mapsto\eta=\sum\o_i\w\g_i$ (Theorem \ref{thm2}). We also prove analogous results for $M$ replaced by a closed subset $X\subset M$ satisfying appropriate conditions (Theorems \ref{thm3} and \ref{thm4}).

For solving in $C^\infty$ category the continuity part of the problem we use results of Doma\'nski and Vogt \cite{DV}. In the case $k=1$ the question can be stated in terms of the exactness and splitting of natural complexes which can be viewed as global Koszul complexes. This case was analyzed in detail in \cite{DJ}.
In order to make the paper accessible for a wider audience we present the results in elementary way.

\section{Algebraic preliminaries, Saito's theorem}

Let $R$ be a commutative ring with unity. Recall that a sequence of elements
\[
a_1,a_2,\dots,a_r \eqno{(\diamond)}
\]
generating a proper ideal of $R$ is called {\it regular} if $a_1$ is a non-zero-divisor of $R$ and $a_i$ is a non-zero-divisor of the quotient ring $R/(a_1R+\cdots+a_{i-1}R)$ for $i=2,\dots,r$.
The sequence $(\diamond)$ is called regular on a $R$-module $\MM$ if  $a_1$ is a non-zero-divisor on $\MM$,  $a_i$ is a non-zero-divisor on the quotient module $\MM/(a_1\MM+\cdots+a_{i-1}\MM)$ for $i=2,\dots,r$, and $a_1\MM+\cdots+a_r\MM\not=\MM$.
The \emph{depth} of a proper ideal $I\subset R$, denoted $\dpt I$, is defined as the supremum of lengths of regular sequences in $I$. For convenience we also define $\dpt\{0\}=0$ and $\dpt\{R\}=\infty$. For properties of the depth one can consult \cite{E} or Appendix 1 in \cite{JZ1}.

\begin{remark}\rm
Note that the depth of a proper ideal $I\subset R$ can not exceed the number of its generators. If $R$ is a local ring with a unique maximal ideal $\mm$, then $I\subset \mm$ and $\dpt{I}$ can not exceed the number of generators of $\mm$. More generally, if $I_1\subset I_2\subset R$ then $\dpt I_1\le \dpt I_2$.
\end{remark}

\begin{remark}\rm
We will use the notion of depth for ideals $I_x^r$ in the rings $C_x^r$ of $C^r$ function germs at given points $x\in M$. Such rings are local with the number of generators of the maximal ideal equal to $n=\dim M$. Thus, $\dpt I_x^r\le n$ for any proper ideal $I_x^r\subset C^r_x$. Recall that in the holomorphic category the depth of an ideal of holomorphic function germs coincides with the codimension of the analytic set germ of zeros of this ideal, see \cite{GH}, Chapter 5.3. In the real analytic category the depth of an ideal coincides with the codimension of the set of its complex zeros since complexification does not change the depth, see \cite{Ru}.
\end{remark}

Consider a free module $\MM$ of rank $m$ over a commutative ring $R$ with unity and let $e_1,\dots,e_m$ be a free basis of $\MM$. We denote by $\w^p\MM$ the $p$-th exterior product of $\MM$ and identify $\w^0\MM=R$ and $\w^{-1}\MM=\{0\}$.
Fix elements $\o_1,\dots,\o_k$ of $\MM$ and consider their exterior product in $\w^k\MM$
\[
\Om=\o_1\w\cdots\w \o_k=\sum_{1\le i_1<\cdots<i_k\le m} a_{i_1\cdots i_k}\,e_{i_1}\w\cdots\w e_{i_k}.
\]
Let $I(\Om)$ denote the ideal in $R$ generated by the coefficients $a_{i_1\cdots i_k}$. Clearly, it is independent of the choice of the basis in $\MM$ and $\dpt I(\Om)\le m-k+1$ as $I(\Om)$ is contained e.g. in the ideal generated by the coefficients of $\o_1$ and even by its first $m-k+1$ coefficients. For proving main results we will use the following theorem proved in Saito \cite{Sa} under additional assumption that $R$ is Noetherian.

\begin{theorem}\label{thmSaito}
If $p<\dpt I(\Om)$ then for any $\eta\in \w^p \MM$ satisfying $\Om\w\eta=0$ there exist $\g_1,\dots,\g_k\in \w^{p-1}\MM$ such that
\[
\eta=\sum_{j=1,\dots,k}\o_j\w\g_j. \eqno(\star)
\]
\end{theorem}

For the proof we need the following lemmata.
\begin{lemma}\label{lemS1}
There exists $n>0$ such that for any $a\in I(\Omega)$ and arbitrary $p=0,1,\dots,m$ and $\eta\in \w^p\MM$ satisfying $\Om\w\eta=0$ there exist $\g_1,\dots,\g_k\in \w^{p-1}\MM$ such that
\[
a^n\eta=\sum \o_j\w\g_j. \eqno{(\star\star)}
\]
\end{lemma}

\begin{proof}
We almost verbatim reproduce the proof of statement (i) of the main theorem in \cite{Sa} to show that the assumption that $R$ is Noetherian invoked there is not needed. The ideal $I(\Om)$ is generated by a finite number of nonzero coefficients $a_I=a_{i_1\cdots i_k}$ of $\Om$, thus it is enough to prove that for each such $a=a_I$ there exists $n$ such that ($\star\star$) holds (the overall $n$ will be larger but finite). We may assume that $a$ is not nilpotent, otherwise there is nothing to prove. Then the set of powers $U=\{a^i\}_{i\ge 0}$, forms a multiplicative subset of $R$ and we may consider the localization $\Ra$ of $R$ over $U$. Its elements can be written as $b/a^i$, $b\in R$. There is a canonical homomorphism $R\to \Ra$ which, as $\MM$ is free, induces homomorphisms $\MM\to\MMa:=\MM\otimes\Ra$ and $\w^p\MM\to\w^p\MMa$ denoted $\o\mapsto [\o]$ and $\eta\mapsto [\eta]$. In particular, we have $R\ni b\mapsto [b]=b/1$ and the homomorphism $\MM\to\MMa$ is induced by the transformation of the basis $e_i\mapsto [e_i]$.

Consider the product $[\Om]=[\o_1]\w\cdots\w[\o_k]$. Since $a$ was chosen one of the coefficients of $\Om$, the form $[\Om]$ written in the free basis $[e_{i_1}]\w\cdots\w[e_{i_k}]$ has $[a]$ as one of its coefficients. Since $[a]$ is invertible in $\Ra$, the sequence $[\o_1],\dots,[\o_k]$ can be completed with some elements $\alpha_1,\dots,\alpha_{m-k}\in\MMa$ to a basis in $\MMa$. Now our problem can be transformed to an easy version. We replace $\eta\in\w^p\MM$ by $[\eta]\in\w^p\MMa$ and look for $\g'_1,\dots,\g'_k\in \w^{p-1}\MMa$ such that
\[
[\eta]=\sum [\o_j]\w\g'_j.
\]
Such $\g'_1,\dots,\g'_k$ exist. To see this one can write $[\eta]$ in the basis which is given by $p$-th wedge products of elements $[\o_i]$ and $\alpha_j$ of the basis in $\MMa$. The coefficients in this representation of $[\eta]$ which appear in elements which consist of $p$-th wedge products of  $\alpha_j$, exclusively, must vanish because $[\Om]\w[\eta]=[\o_1]\w\cdots\w[\o_k]\w[\eta]=0$. The remaining elements in the sum contain wedge products with at least one $[\o_i]$, which proves that $\g'_j$ giving the above representation of $[\eta]$ exist.

We may write $\g'_j=a^{-n_1}[\bar\g_j]$ with $\bar\g_j\in \w^{p-1}\MM$ and a common $n_1\ge 0$. Then
\[
\left[a^{n_1}\eta - \sum \o_j\w\bar\g_j\right]=a^{n_1}[\eta]-\sum [\o_j]\w[\bar\g_j]=0
\]
and, by definition of the localization over powers of $a$, there exists $n_2\ge 0$ such that
\[
a^{n_2}\left(a^{n_1}\eta - \sum \o_j\w\bar\g_j\right)=0\ \ \ \text{in}\ \ \w^p\MM.
\]
Taking $n=n_1+n_2$ and $\g_j=a^{n_2}\bar\g_j$ we obtain ($\star\star$).
\end{proof}
%\newpage

Consider a sequence ($\diamond$) of elements in $R$. We introduce ideals of $R$,
\[
R_1=\{0\}, \ \ R_i=a_2R+\cdots+a_iR,\ \  i=2,\dots,r, \ \ \text{and}\ \  Q_i=a_1R+R_i.
 \]
For elements $a,b\in R$ and an ideal $I\subset R$ we shall often write $a=b$ mod $I$ or $a=b$ in $R/I$ instead introducing special notation for equivalence classes when considering elements in the quotient ring $R/I$.  We need two elementary properties.

\begin{lemma}\label{lemS2}
If the sequence ($\diamond$) is regular then $a_1$ is a \nzd on $R/R_i$, $i=1,\dots,r$.
\end{lemma}

\begin{proof}
We will use induction with respect to $i$. For $i=1$ the statement follows from the assumption. Suppose that it holds for an $i<r$. Given $b_1\in R$, we should prove that  $a_1b_1=0$ in $R/R_{i+1}$ implies that  $b_1=0$ in $R/R_{i+1}$. Let $a_1b_1=0$ in $R/R_{i+1}$, i.e., $a_1b_1\in R_{i+1}$. Then
\[
a_1b_1=a_2b_2+\cdots+a_{i+1}b_{i+1}, \ \ \ b_2,\dots,b_{i+1}\in R, \eqno{(\diamond\diamond)}
\]
which means that $a_{i+1}b_{i+1}\in Q_i$ and, equivalently,  $a_{i+1}b_{i+1}=0$ in $R/Q_i$. Since the sequence $(\diamond)$ is assumed regular, we deduce that $b_{i+1}=0$ in $R/Q_i$ which means that $b_{i+1}=c_1a_1+\cdots+c_ia_i$. Plugging such $b_{i+1}$ to ($\diamond\diamond$) we find that
\[
a_1(b_1-c_1a_{i+1})=a_2(b_2+c_2a_{i+1})+\cdots+a_i(b_i+c_ia_{i+1}).
\]
Thus, $a_1(b_1-c_1a_{i+1})=0$ in $R/R_i$ and from the induction assumption we deduce that $b_1-c_1a_{i+1}=0$ in $R/R_i$. This means that $b_1\in c_1a_{i+1}+R_i$, consequently $b_1\in R_{i+1}$, i.e., $b_1=0$ in $R/R_{i+1}$ which was to be proved.
\end{proof}

\begin{lemma}\label{lemS3}
If ($\diamond$) is regular then, for any $n\ge 1$, the sequence $a_1^n, a_2,\dots,a_r$ is regular.
\end{lemma}

\begin{proof} We use induction with respect to $n$. The case $n=1$ is tautological. Assume that the assertion holds for all powers of $a_1$ less then $n$. To prove that $a_1^n, a_2,\dots,a_r$ is regular note that $a_1^n$ is \nzd in $R$ since $a_1$ is. We have to prove that $a_{i+1}$ is \nzd on $R/(a_1^nR+R_i)$, for $i=1,\dots,r-1$. This is equivalent to showing that $a_{i+1}b\in a_1^nR$ mod $R_i$ implies $b\in a_1^nR$ mod $R_i$.

Let $a_{i+1}b\in a_1^nR$ mod $R_i$, i.e., $a_{i+1}b=a_1^nc=(a_1)^{n-1}a_1c$ mod $R_i$. These equalities and the fact that $a_{i+1}$ is \nzd on $R/(a^{n-1}R+R_i)$ (the induction assumption) imply that $b=a_1^{n-1}d$ mod $R_i$. Plugging such $b$ to the above equality we get $a_1^{n-1}(a_{i+1}d-a_1c)=0$ mod $R_i$. From Lemma \ref{lemS2} it follows that $a_1$ and thus $a_1^{n-1}$ is a \nzd on $R/R_i$, therefore $a_{i+1}d=a_1c$ mod $R_i$. In consequence $a_{i+1}d=0$ mod $a_1R+R_i$ and by the regularity of the sequence ($\diamond$) we find that $d\in a_1R+R_i$, i.e., $d=a_1e$ mod $R_i$. Plugging such $d$ to $b=a_1^{n-1}d$ mod $R_i$ we see that $b\in a_1^nR$ mod $R_i$ which ends the proof.
\end{proof}

{\it Proof of Theorem \ref{thmSaito}.}
We will modify the proof in \cite{Sa} using the above lemmata instead of assuming that $R$ is Noetherian.
The proof uses double induction with respect to $(p,k)$ where $p\ge 0$, $k\ge 0$, $(p,k)\not=(0,0)$.

Denote by $S(p,k)$ the statement of the theorem with fixed $p$ and $k$. In the case $k=0$ we artificially define $\Omega=1\in \w^0\MM=R$ and assume the sums in ($\star$) and ($\star\star$) being zero. Then $\eta=0$ and statement $S(p,0)$ holds for any $p\ge 1$. In the case $p=0$, $k\ge 1$ we have $\w^0\MM=R$ and $\eta\in R$. Since $\dpt I(\Omega)>p=0$, the ideal generated by the coefficients of $\Om$ contains a \nzd and  $\Om\w \eta=0$ implies that $\eta=0$, i.e. statement (i) holds.

Assume now that $k>0$ and $0<p<\dpt I(\Om)$. We have to prove the statement $S(p,k)$ assuming that it holds for $(p-1,k)$ and $(p,k-1)$.

Let $a_1,\dots,a_r$ be a regular sequence of elements in $I=I(\Om)$, with $r=\dpt I(\Om)$.
Consider $\eta\in\w^p\MM$ such that $\Om\w\eta=0$.  Since $a_1\in I(\Om)$, from Lemma \ref{lemS1} it follows that
\[
(a_1)^n\eta=\sum \o_j\w\g_j \eqno{(\star\star)}
\]
for some $n\ge 1$ and $\g_j\in\w^{p-1}\MM$. Lemma \ref{lemS3} implies that the sequence $a_1^n,a_2,\dots,a_r$ is also regular, in particular the element $a_1^n\in I$ is a \nzd in $R$.

We denote $a=a_1^n$ and introduce the quotient ring $\bR=R/aR$. All elements $b$ of $R$ have their canonical counterparts $\bar b=b+aR$ in $\bR$. Similarly, elements of the free modules $\MM$ and $\w^p\MM$ over $R$ have their canonical counterparts in corresponding free modules $\bMM=\MM/a\MM$ and $\w^p\bMM\simeq \w^p\MM/a\w^p\bMM$ over $\bR$, also denoted with bars. The regular sequence $a_1^n,a_2,\dots,a_r$ in $R$ will now be replaced by the regular sequence $\ba_2,\dots,\ba_r$ in $\bR$.

Passing to the quotients in the equality ($\star\star$) we obtain
\[
\sum\bo_j\w\bg_j=0.
\]
Multiplying this equality by $\bo_1\w\cdots\w{\hat\bo_j}\w\cdots\w\bo_k$ (where $\bo_j$ is omitted) gives $\bg_j\w\bo_1\w\cdots\w\bo_k=0$. This allows us to use the induction assumption for $\bg_j\in \w^{p-1}\bMM$ playing the role of $\eta$ in ($\star$), all over the quotient ring $\bR=R/aR$ where the sequence $\ba_2,\dots,\ba_r$ is regular. Namely, the elements $\ba_i$ belong to the ideal generated by the coefficients of $\bO=\bo_1\w\cdots\w\bo_k$. Since the above sequence is regular we have $\dpt I(\bO)\ge r-1$ and, since $p-1<r-1$, we can apply statement $S(p-1,k)$ (true by the induction assumption) and get
\[
\bg_j=\sum_i\bo_i\w\bxi_{ji}
\]
for some $\bxi_{ji}\in \w^{p-2}\bMM$. This equation can be lifted back to $\MM$, i.e., there exist $\xi_{ji}\in \w^{p-2}\bMM$ and $\zeta_j\in\w^{p-1}\MM$ such that $\bxi_{ji}=\xi_{ji}+a\w^{p-2}\MM$ and
\[
\g_j=\sum_i\o_i\w\xi_{ji}+a\zeta_j.
\]
Plugging $\g_j$ to the relation ($\star\star$) and taking into account that $a=a_1^n$ we obtain
\[
a(\eta-\sum_j\o_j\w\zeta_j)=\sum_{i,j}\o_j\w\o_i\w\xi_{ji}.
\]
If $k=1$ then the right hand side is zero and, as $a$ is a \nzd in $R$, we get $\eta=\sum\o_j\w\zeta_j$. For $k>1$ we multiply both sides by $\o_2\w\cdots\w\o_k$ which gives
\[
a(\eta-\sum_j\o_j\w\zeta_j)\w\o_2\w\cdots\w\o_k=0.
\]
Again $a$ is a \nzd in $R$, thus
\[
(\eta-\sum_j\o_j\w\zeta_j)\w\o_2\w\cdots\w\o_k=0.
\]
The ideal $\hat I(\hat \Om)$ generated by the coefficients of $\hat \Om=\o_2\w\cdots\w\o_k$ contains the ideal $I(\Om)$, thus
$\dpt\hat I(\hat\Om)\ge\dpt I(\Om)=r$. Therefore,  we can use the  induction assumption for $(p,k-1)$ and obtain
\[
\eta-\sum_{j=1,\dots,k}\o_j\w\zeta_j=\sum_{i=2,\dots,k}\o_i\w\theta_i
\]
for some $\theta_i\in\w^{p-1}\MM$. This shows that $\eta$ can be written in the form ($\star$) and completes the proof.
\hfill $\Box$

\begin{remark} \rm
If $R$ is Noetherian then, in the above proof, it is not necessary to choose a priori the regular sequence $a_1,\dots,a_r$ in $I(\Om)$ and use the lemma to assure that $a_1^n,\dots,a_r$ is regular. Namely, in this case any \nzd in $I(\Om)$ can be completed to a regular sequence in $I(\Om)$ with $r=\dpt I(\Om)$, in particular $a_1^n$ can be so completed as it is implicitly used in the proof in \cite{Sa}.
\end{remark}

\section{Main results.}
\subsection{Bundles over $M$}

In our considerations we will fix a category $C^r$ which be any of the three categories: $C^{\infty}$ (smooth), $C^{\omega}$ (real analytic), $C^{hol}$ (holomorphic). Let $E$ denote a vector bundle of class $C^r$ of rank $m$ over a manifold $M$. In the first two cases we assume that $M$ is a second countable real differential manifold of class $C^r$. In the last case we assume that $M$ is a complex Stein manifold (for example, a pseudoconvex domain in $\C^n$).

We denote by $\Lambda^r_p(M;E)$ the linear space of $C^r$ sections of the $p$-th skew-symmetric power of $E$ denoted $\w^pE$. In particular, $\Lambda^r_0(M)=\Lambda^r_0(M;E)$ is the space of $C^r$ functions on $M$, also denoted by $C^r(M)$. Note that $C^r(M)$ is also a ring and $\Lambda^r_p(M;E)$ are modules over $C^r(M)$. Local sections in $\Lambda^r_p(M;E)$ of the bundle $\w^pE$ can be written in the form
\[
\eta=\sum_{1\le i_1<\cdots<i_p\le m}g_{i_1\cdots i_p}\,e_{i_1}\w\cdots\w e_{i_p},
\]
where $"\w"$ denotes the pointwise skew symmetric product,  $g_{i_1\cdots i_p}$ are local functions of class $C^r$ on $M$ and $e_1,\dots e_m$ are local, pointwise linearly independent sections of $E$ of class $C^r$.

Let us fix $1\leq k\leq m$ and sections  $\o_1,\dots,\o_k$ of $E$ of class $C^r$. Denote
\[
\Om=\o_1\w\cdots\w\o_k.
\]
The sections $\o_1,\dots,\o_k$ define two linear operators $A$, $B$ which form a diagram
\begin{equation}\label{cABM}
(\La^r_{p-1}(M;E))^k \mathop{\longrightarrow}\limits^{A} \La_p^r(M;E) \mathop{\longrightarrow}\limits^{B} \La_{p+k}^r(M;E)
\end{equation}
where
\begin{equation}\label{opA2}
A(\g_1,\dots,\g_k)=\sum_{1\le j\le k}\o_j\w\g_j,\quad
\end{equation}
\begin{equation}\label{opB2}
B\eta=\Om\w\eta.
\end{equation}
Clearly, $BA=0$ thus we have $\im A\subset \Ker B$ and \eqref{cABM} is a complex. Note that $A$ and $B$ are also homomorphisms of the corresponding modules. Our first problem is to determine when $\im A=\Ker B$, i.e., when the complex \eqref{cABM} is \emph{exact}.
Explicitly, we ask when for any $\eta\in \La^r_p(M;E)$ satisfying $\Om\w\eta=0$ there exist $\g_1,\dots,\g_k\in \La^r_{p-1}(M;E)$ such that
\begin{equation}\label{star2}
\eta=\sum_{1\le j\le k}\o_j\w\g_j.
\end{equation}

In order to state a sufficient condition for this property we define the \emph{singular set} $\Sing(\Om)$ as the set of points where $\o_1,\dots,\o_k$ are linearly dependent or, equivalently,
\[
\Sing(\Om)=\{x\in M\ |\ \Om(x)=0\}.
\]
For a fixed $x\in M$ and a local basis near $x$ of $C^r$ sections $e_1,\dots,e_m$ of the bundle $E$ we can write
\[
\Om=\sum_{1\le i_1<\cdots<i_k\le m}f_{i_1\cdots i_k}\,e_{i_1}\w\cdots\w e_{i_k}.
\]
The \emph{ideal} $I^r_x(\Om)$ is defined as the ideal of germs at $x$ of functions of class $C^r$ generated by the coefficients $f_{i_1\cdots i_k}$ of the above expansion. If $x\not\in \Sing(\Om)$ then this ideal coincides with the ring $C_x^r(M)$ of function germs at $x$ of class $C^r$, otherwise it is a proper ideal of $C_x^r(M)$ independent of the choice of a local basis in $E$.

\begin{theorem}\label{thm1}
In any of the categories $C^r$, $r\in\{\omega, hol,\infty\}$, the complex \eqref{cABM} is exact if
\begin{equation}\label{depth1}
p<\dpt I^r_x(\Om)\ \ \forall\, x\in \Sing(\Om).
\end{equation}
\end{theorem}

In particular, if in the above framework $\eta$ is a section of $E$ and $\dpt I^r_x(\Om)\ge 2$ for any $x\in \Sing(\Om)$ then $\o_1\w\cdots\w\o_k\w\eta=0$ is a sufficient (and necessary) condition that there are functions $f_i$ on $M$ such that $\eta=\sum f_i\,\o_i$.

\begin{remark} \rm
Clearly, if $M$, $E$ are real analytic and $\o_i$ are real analytic sections of $E$ then they define linear operators $A$, $B$ (and the ideals $I^r_x(\Om)$) in both categories $C^\infty$ and $C^\omega$.  In this case verifying the depth condition (\ref{depth1}) in the ring of analytic function germs may be easier since such ring is Noetherian and has no zero divisors. Our continuity results will use exactness of the complex \eqref{cABM} in the $C^\infty$ category. Thus, the following corollary following from our proofs may be useful in this context.
\end{remark}

\begin{corollary}\label{cor1}
If all data: $M$, $E$, $\o_1,\dots,\o_k$, and the depth condition \eqref{depth1}, are considered in the category $C^\omega$, then
the complex \eqref{cABM} is also exact in the smooth category.
\end{corollary}

\begin{remark} \rm
For $k=1$  the first statement of the above result and the first statements in Theorems \ref{thm2}, \ref{thm3} and \ref{thm4} were proven in \cite{DJ}. If $x\in \Sing(\Om)$ then the theorem holds locally, i.e. for germs at $x\in M$, if $p<\dpt I^r_x(\Om)$ (Lemma \ref{lem-germs}). If $x\not\in \Sing(\Om)$ then $\o_1,\dots,\o_k$ are pointwise linearly independent locally around $x$ and the local version is classical and can be shown in elementary way. If $\Sing(\Om)$ is empty then condition (\ref{depth1}) automatically holds. Then the only difficulty of passing from local exactness to global in the analytic cases reduces to using elementary results from the theory of analytic sheaves.
\end{remark}

The continuity part of our problem can be solved in the $C^\infty$ category as follows. For a fixed $p$ consider the complex \eqref{cABM} in the $C^\infty$ category,
\begin{equation}\label{cABMs}
(\La^\infty_{p-1}(M;E))^k \mathop{\longrightarrow}\limits^{A} \La_p^\infty(M;E) \mathop{\longrightarrow}\limits^{B} \La_{p+k}^\infty(M;E).
\end{equation}
All three spaces are spaces of smooth sections of vector bundles thus they have natural $C^\infty$ topologies which make them Fr\'echet spaces.
Recall that a continuous linear operator $A:X\to Y$ between linear locally convex topological spaces \emph{splits} if its image is closed in $Y$ and there is a linear continuous right inverse operator $A^+:\im A\to X$ (i.e. satisfying $AA^+=\Id\vert_{\im A}$), where $\im A$ has the topology induced from $Y$. Similarly, one says that a complex $X\mathop{\longrightarrow}\limits^{A} Y \mathop{\longrightarrow}\limits^{B}Z$ of such spaces \emph{splits}, with $A$ and $B$ continuous, if $\im A=\Ker B$ and $A$ splits.

The following result holds in the $C^\infty$ topology.

\begin{theorem}\label{thm2}
If $M$, $E$ and $\o_1,\dots,\o_k$ are smooth and the complex \eqref{cABMs} is exact then it splits for any $p\ge 1$. In particular, it splits if $p<\dpt I_x^\infty(\Om)$, $x\in \Sing(\Om)$.
\end{theorem}

The above theorem and Corollary \ref{cor1} imply that the complex \eqref{cABMs} also splits if all data are real analytic and $p<\dpt I^\omega_x(\Om)$, $x\in \Sing(\Om)$.

Note that for $p>\rank E$ the complex \eqref{cABMs} and Theorem \ref{thm2} are trivial.

\begin{remark} \rm
In the holomorphic category one can use the compact-open topology in the spaces of holomorphic functions and holomorphic sections of the bundle $E$ over a Stein manifold $M$, which makes them Fr\'echet spaces. In this case the technique presented in \cite{DJ}, Theorem 2.3, allows to obtain results analogous to the above theorem assuming additionally that $M$ has finitely many connected components and is strongly pseudoconvex or it satisfies the strong Liouville property.
\end{remark}

\begin{problem}[cf. \cite{DJ}, Problem 2.4]
Is the same true for any Stein manifold?
\end{problem}

\subsection{Bundles over $X$}

The same results can be stated for sections of analogous bundles over a closed subset $X$ of $M$. As before, if not stated otherwise, the manifold $M$, the bundle $E$, and functions, sections etc., are considered in a fixed category $C^r$ which is $C^\infty$, $C^\omega$ or $C^{hol}$. In the real analytic and holomorphic categories we additionally assume that $X$ is an analytic subset of $M$, i.e. it is locally defined as the set of zeros of a family of real analytic (resp. holomorphic) functions.

Here it will be more convenient to treat $\Lambda^r_p(M;E)$ as the module over the ring of functions $C^r(M)$. Denote by $C^r(M,X)$ the subring of $C^r(M)$ of $C^r$ functions vanishing on $X$ and let  $\Lambda^r_p(M,X;E)$ denote the module over $C^r(M)$ of $C^r$ sections of the $p$-th skew-symmetric power of $E$ which vanish on $X$. In particular, $C^r(M,X)=\Lambda^r_0(M,X;E)$. The spaces $\Lambda^r_p(M,X;E)$ are, in a natural way, modules over the ring $C^r(M)$.  Define the quotient rings
\begin{equation}\label{quotient ring}
C^r(X)=C^r(M)/C^r(M,X)=:\Lambda^r_0(X;E)
\end{equation}
and the quotient modules over $C^r(M)$, and also over $C^r(X)$,
\begin{equation}\label{quotient module}
\Lambda^r_p(X;E)=\Lambda^r_p(M;E)/\Lambda^r_p(M,X;E), \quad p\ge 1.
\end{equation}
Elements of $\Lambda^r_p(X;E)$, called $C^r$ \emph{sections} $\o:X\to \w^pE$, can be written in the form $\o_M+\Lambda^r_p(M,X;E)$, where $\o_M:M\to \w^pE$ is a section of class $C^r$. The exterior product factorizes to the quotient spaces \eqref{quotient module}. For any $x\in X$ the value $\o(x)$ in the fiber $E_x$ of $E$ is well defined. Similarly, one can write germs or local sections. In particular, a local basis $e_1,\dots,e_m$ of $E$ on $X$ is understood as a collection of pointwise linearly independent local sections $M\to E$ of class $C^r$, each taken modulo local elements in $\Lambda^r_1(M,X;E)$. Again, locally in $X$, the elements of $\Lambda^r_p(X;E)$ can be written as
\[
\eta=\sum_{1\le i_1<\cdots<i_p\le m}g_{i_1\cdots i_p}\,e_{i_1}\w\cdots\w e_{i_p},
\]
where  $g_{i_1\cdots i_p}$ are locally defined elements of $C^r(X)$.

Fix $1\leq k\leq m$ and consider elements  $\o_1,\dots,\o_k\in \Lambda^r_1(X;E)$. Again, we denote
\[
\Om=\o_1\w\cdots\w\o_k
\]
and the value $\Om(x)=\o_1(x)\w\cdots\w\o_k(x)$ is well defined for $x\in X$. The elements $\o_1,\dots,\o_k$ define linear operators $A_X$, $B_X$ in the diagram
\begin{equation}\label{opABXc}
(\La^r_{p-1}(X;E))^k\ \mathop{\longrightarrow}^{A_X}\ \La^r_p(X;E)\ \mathop{\longrightarrow}^{B_X}\ \La_{p+k}^r(X;E),
\end{equation}
where
\begin{equation}\label{opABX}
A_X(\g_1,\dots,\g_k)=\sum \o_j\w\g_j\ \ \mathrm{and}\ \  B_X\eta=\Om\w\eta.
\end{equation}
As before, $B_XA_X=0$ thus $\im A_X\subset \Ker B_X$ and the diagram defines a complex. We may ask the question when it is exact, i.e., $\im A_X=\Ker B_X$.

The \emph{singular set} of $\Om$, $\Sing(\Om)=\{x\in X\ |\ \Om(x)=0\}$ is now a subset of $X$. For a fixed $x\in X$ we can choose a basis of $C^r$ sections $e_1,\dots,e_m:X\to E$ defined in a neighbourhood of $x$ and write
\[
\Om=\sum_{1\le i_1<\cdots<i_k\le m}f_{i_1\cdots i_k}\,e_{i_1}\w\cdots\w e_{i_k},
\]
where the coefficients are locally defined elements of $C^r(X)$.

Analogously to \eqref{quotient ring} define the quotient ring $C^r_x(X)=C^r_x(M)/C^r_x(M,X)$, where $C^r_x(M)$ is the ring of germs at $x$ of local functions of class $C^r$ on $M$ and $C^r_x(M,X)$ the ring of germs of such functions vanishing on $X$. We denote by
\begin{equation}\label{local ideal}
I^r_{X,x}(\Om)\subset C^r_x(X)
\end{equation}
the ideal of $C^r_x(X)$ generated by the coefficients $f_{i_1\cdots i_k}$ of the above expansion. (If $x\not\in \Sing(\Om)$ then this ideal coincides with the ring $C_x^r(X)$.)

\begin{theorem}\label{thm3}
Let $X\subset M$ be a closed subset which is real analytic and coherent, in the case of real analytic category, and complex analytic in the case of holomorphic category. Then in all three categories the complex \eqref{opABXc} is exact for each $p\ge 1$ satisfying
\begin{equation}\label{depth2}
p< \dpt I^r_{X,x}(\Om)\ \ \forall\, x\in \Sing(\Om)\subset X.
\end{equation}
\end{theorem}

Recall that a sheaf $\FF$ of modules over a sheaf of rings $\RR$ is called \emph{coherent} if it satisfies two conditions: (a) it is of finite type (i.e., locally it is generated by a finite family of sections); (b) given any finite family of sections $s_1,\dots,s_k$ of $\FF$ over an open subset $U\subset M$, the sheaf of relations, i.e. of elements $(f_1,\dots,f_k)\in (\RR(U))^k$ such that $f_1s_1+\cdots+f_ks_k=0$, is of finite type.

An analytic subset $X\subset M$ of a real analytic (respectively, holomorphic, Stein) manifold $M$ is called {\it coherent} if the sheaf $\OO(M,X)$ of real analytic (respectively, holomorphic) function germs  vanishing on $X$ is coherent (over the sheaf $\OO$). This is equivalent to say that the sheaf of function germs vanishing on $X$ is of finite type (the sheaf of relations is always of finite type). All analytic subsets of Stein manifolds as well as every real analytic manifold are coherent by a theorem of Oka (see \cite{C}, Proposition 4). For more information on coherence and on real analytic coherent sets see \cite{N}.

As earlier, the following fact may be helpful.

\begin{corollary}\label{cor2}
If in Theorem \ref{thm3} the assumptions hold in the real analytic category, together with the depth condition \eqref{depth2}, then the complex \eqref{opABXc} is also exact in the smooth category.
\end{corollary}

For solving the splitting problem in $C^\infty$ category we consider the complex
\begin{equation}\label{cABX}
(\La^\infty_{p-1}(X;E))^k\ \mathop{\longrightarrow}\limits^{A_X}\ \La_p^\infty(X;E)\ \mathop{\longrightarrow}\limits^{B_X}\ \La_{p+k}^\infty(X;E).
\end{equation}
All three spaces above have well defined $C^\infty$ topologies and are Fr\'echet spaces as quotients of Fr\'echet spaces with $C^\infty$ topologies. Clearly, the operators $A_X$ and $B_X$ given by \eqref{opABX} are linear and continuous.

Recall that a closed subset $X\subset M$ has the \emph{extension property} (in the category $C^\infty$) if there exists a linear continuous operator $S:C^\infty(X)\to C^\infty(M)$ such that $S(f)\vert_X=f$ or, more precisely, the natural quotient map $Q:C^\infty(M)\to C^\infty(X)$ has a continuous linear right inverse. If $M$ is a real analytic manifold then all closed semianalytic subsets $X\subset M$ as well as all coherent closed subanalytic subsets have the extension property, see \cite{BS}. In Sections 0 and 7 of \cite{BS} the reader can find more general classes of sets with the extension property, as well as an example where the property fails.

\begin{theorem}\label{thm4}
If $M$ and $E$ are smooth, $\o_1,\dots,\o_k\in \La_1^\infty(X;E)$, the closed subset $X\subset M$ has the extension property, and the complex \eqref{cABX} is exact for a $p\ge 1$ then it splits. In particular, it is exact and splits if $p<\dpt I_x^\infty(\Om)$ for $x\in \Sing(\Om)$.
\end{theorem}

\begin{corollary}\label{cor3}
\noindent If $M$ and $E$ are $C^\omega$, $\o_1,\dots,\o_k\in \La_1^\omega(X;E)$, and the set $X$ is a closed coherent analytic subset of $M$, then the complex \eqref{cABX} is exact and it splits for any $p\ge 1$ satisfying $p<\dpt I_{X,x}^\omega(\Om)$ for all $x\in \Sing(\Om)\subset X$.
\end{corollary}

This corollary follows directly from the above theorem and Corollary \ref{cor2} since a closed coherent analytic subset of analytic $M$ has the extension property.

\begin{remark}\rm
If $E$ is a trivial bundle, $M$, $E$, $\o_1,\dots,\o_k$ are real analytic, and $X\subseteq M$ is a closed semicoherent subanalytic (in particular, semianalytic or subanalytic and coherent) subset then the operator $A_X$ in \eqref{cABX} has a continuous linear right inverse even if the complex is  not exact (see \cite{BS}, Theorem 0.1.3 and Corollary 8.4).
\end{remark}

\section{Local and global exactness of sheaf complexes}\label{s-sheaves}

For proving Theorems \ref{thm1} and \ref{thm2} we need an elementary result in sheaf theory. Consider a metrizable topological space $X$ which has the Lindel\"of property, i.e. any open cover of $X$ has a countable subcover.  Let $\FF$, $\GG$, and $\HH$ be sheaves of modules over $X$ and consider a complex
\begin{equation}\label{c-sheaves}
\FF \mathop{\longrightarrow}\limits^A \GG \mathop{\longrightarrow}\limits^B \HH,
\end{equation}
where $A$ and $B$ are morphisms of sheaves of modules. Such a complex is called exact if the corresponding complex of fiber morphisms
\begin{equation}\label{c-fibers}
\FF_x \mathop{\longrightarrow}\limits^{A_x} \GG_x \mathop{\longrightarrow}\limits^{B_x} \HH_x
\end{equation}
is exact for any $x\in X$. Recall that any fiber, say $\FF_x$, can be identified with the module of germs at $x$ of local sections of $\FF$ (over a ring $R_x$).

The morphisms $A$ and $B$ define homomorphisms of modules of global sections
\begin{equation}\label{c-modules}
\FF(X) \mathop{\longrightarrow}\limits^{A_X} \GG(X) \mathop{\longrightarrow}\limits^{B_X} \HH(X),
\end{equation}
which also forms a complex, i.e., $B_XA_X=0$.

Recall that a sheaf of rings $\RR$ (respectively, o sheaf of modules $\FF$)  is called \emph{soft} if for any closed subset $S\subset X$ the restriction map $\RR(X)\to \RR(S)$ (resp. $\FF(X)\to \FF(S)$) is surjective. All rings will be assumed commutative with unity. If $X$ is metrizable then the following separability condition is sufficient for softness of $\RR$  and, more generally, for softness of any sheaf of modules over $\RR$, (cf. \cite{GR}, Chapter A, $\S$ 4, Theorem 2).

\medskip
\it Property (P). For any closed subset $S\subset X$ and any open neighbourhood $W$ of $S$ in $X$ there is a global section $\varphi\in \RR(X)$ such that $\varphi\vert_{S}=1$ and $\varphi\vert_{X\setminus W}=0$. \rm
\medskip

Note that any smooth manifold $M$ and its sheaf of germs of $C^\infty$ functions has this property, as well as any closed subset $X$ and its sheaf of rings defined by the sections in $C^\infty(X\cap U)=C^\infty(U)/C^\infty(U,X\cap U)$, for open subsets $U\subset M$.

We will say that a sheaf of unital rings $\RR$ on $X$ has \emph{the partition of unity property} if for any open cover $\{U_i\}$ of $X$ there is a family of global sections $\varphi_i\in\RR(X)$ which satisfies: (i) $\supp\varphi_i\subset U_i$; (ii) it is locally finite, i.e. for any $x\in X$ there is a neighbourhood of $x$ in which all but finitely many $\varphi_j$ vanish; (iii) $\sum_i\varphi_i(x)=1_x$ for all $x\in X$ (here $1_x$ denotes the unit in the fiber ring $R_x$).

In the analytic categories the partition of unity can not be used. Instead, we will use the following properties of a sheaf of modules $\FF$ over a sheaf of rings $\RR$.

\medskip
\it Property (A).
For any $x\in X$ and any germ $s_x\in \FF_x$ there is a finite number of global sections $s_1,\dots,s_q\in \FF(X)$ such that $s_x=\sum a_is_{i,x}$, where $s_{i,x}$ are values (equivalently, germs) at $x$ of $s_i$ and $a_i$ are elements of $R_x$. \rm

\medskip
\it Property (B).
The first cohomology group $H^1(\FF)$ of $\FF$ vanishes.
 \rm \medskip

\begin{proposition}\label{prop1}
(a) If the sheaf $\RR$ has the partition of unity property or the subsheaf $\Ker A\subset \FF$ has the property (B) then exactness of the complexes of modules of germs \eqref{c-fibers} at any $x\in X$ implies exactness of the complex of global modules \eqref{c-modules}.

\noindent (b) If the sheaf of rings $\RR$ has the property (P) or the subsheaf $\Ker B\subset \GG$ has the property (A) then exactness of the complex of modules \eqref{c-modules} implies exactness of the complexes of modules \eqref{c-fibers}
\end{proposition}

Even if the proposition is hardly novel, we present its detailed proof since it provides additional insight into our results.

\begin{proof}
We first prove that exactness of $\eqref{c-fibers}$ implies exactness of $\eqref{c-modules}$. Suppose that  $\beta\in \Ker B_X$. We have to prove that there exists $\alpha\in \FF(X)$ such that
\begin{equation}\label{global-ba}
\beta=A_X\alpha.
\end{equation}
Let $\beta_x$ denote the germ of $\beta$ at $x\in X$. Then $\beta_x\in \Ker B_x$ and exactness of $\eqref{c-fibers}$ implies that there is a germ $\alpha_x$ such that
\begin{equation}\label{germs-ba}
\beta_x=A_x\alpha_x.
\end{equation}
Having $\alpha_x$ at any $x\in X$ we will construct a global $\alpha$ satisfying \eqref{global-ba}. Each equality of germs \eqref{germs-ba} can be replaced by an equality of sections defined in neighbourhoods $U(x)$ of $x$, $\beta\vert_{U(x)}=A\alpha_{U(x)}$, where $\alpha_{U(x)}$ is a local section which represents the germ $\alpha_x$. Since $X$ has the Lidel\"of property, from the open cover $\{U_x\}_{x\in X}$ we may select a countable subcover $\{U_i\}$ of $X$ and, denoting $\alpha_i=\alpha_{U_i}$, we may write
\begin{equation}\label{family-ba}
\beta\vert_{U_i}=A\alpha_i.
\end{equation}

There is no reason that $\alpha_i$ and $\alpha_j$ coincide on $U_i\cap U_j$. In the case of the sheaf of rings $\RR$ having the partition of unity property we may choose a partition subordinate to the cover $\{U_i\}$ which is a locally finite family $\{\varphi_i\}$ of sections of $\RR$ such that $\sum\varphi_i=1_X$, where $1_X$ is the unity section in $\RR(X)$. Define
\[
\alpha=\sum_i\varphi_i\alpha_i.
\]
Such $\alpha$ has the desired property as
\[
A_X\alpha=A_X(\sum\varphi_i\alpha_i)=\sum\varphi_iA_X\alpha_i=\sum\varphi_i\beta\vert_{U_i}=\beta.
\]

If there is no partition of unity we proceed differently. In order to suitably modify $\alpha_i$ we denote $U_{ij}=U_i\cap U_j$ and define the cochain of local sections $\g_{ij}:U_{ij}\to \FF$,
\[
\g_{ij}(x)=\alpha_i(x)-\alpha_j(x), \ \ x\in U_{ij}.
\]
It follows from \eqref{family-ba} that $\g_{ij}:U_{ij}\to \Ker A$, thus $\g_{ij}$ are local sections of the subsheaf $\Ker A\subset \FF$.
The cochain is a cocycle since $\g_{ij}+\g_{jk}+\g_{ki}=0$ on $U_i\cap U_j\cap U_k$. It follows from property (B) of the sheaf $\Ker A$ that, after possibly going to a subcover, there exists a coboundary which consists of sections $\sigma_i:U_i\to \Ker A$ such that
\[
\g_{ij}(x)=\sigma_i(x)-\sigma_j(x), \ \ x\in U_i\cap U_j.
\]
Define $\bar\alpha_i:U_i\to \FF$,
\[
\bar\alpha_i=\alpha_i-\sigma_i.
\]
Then for $x\in U_i\cap U_j$ we have
\[
\bar\alpha_i-\bar\alpha_j=\alpha_i-\alpha_j-\sigma_i+\sigma_j=\g_{ij}-\g_{ij}=0.
\]
This means that all $\bar\alpha_i$ can be glued together to a global section $\bar\alpha$ of $\FF$. Moreover, since $\sigma_i$ take values in $\Ker A$, we have that $A\bar\alpha_i=A\alpha_i$. Therefore, $A\bar\alpha\vert_{U_i}=A\bar\alpha_i=A\alpha_i=\beta_{U_i}$, thus so constructed global section $\bar\alpha$ has the desired property \eqref{global-ba}.

In order to prove the second statement consider a germ $\beta_x\in \GG_x$ such that $B_x\beta_x=0$. If the sheaf of rings $\RR$ has the property (P) then the sheaf $\Ker B$ is soft and then there exists a global section $\beta$ of $\Ker B$ such that its germ at $x$ coincides with $\beta_x$. Otherwise we can use the property (A) of the subsheaf $\Ker B\subset \GG$ which implies that there are global sections $s_1,\dots,s_q$ of $\Ker B$ such that $\beta_x=\sum a_is_{i,x}$, where $s_{i,x}$ are germs at $x$ of $s_i$ and $a_i$ are elements of $R_x$. Then $\beta=\sum a_is_i$ is a global section of $\Ker B$ and its germ at $x$ coincides with the given germ $\beta_x$. In both cases we have global section $\beta$ of $\Ker B$ which has the given germ $\beta_x$ at $x$. It now follows from exactness of the complex \eqref{c-modules} that there is a global section $\alpha$ of $\FF$ such that $\beta=A_X\alpha$. This implies that $A_x\alpha_x=\beta_x$ and shows exactness of the complex \eqref{c-fibers}.
\end{proof}

\section{Proofs of main results}\label{s-proofs}

Let us fix a category $C^r$, $r\in\{\infty, \omega, hol\}$, and a point $x\in M$.  The operators defined in \eqref{opABXc}, \eqref{opABX} have their germ versions $A_{X,x}$, $B_{X,x}$ and the germ version of the complex has the form
\begin{equation}\label{cABX-germs}
(\La^r_{p-1}(X,x;E))^k\ \mathop{\longrightarrow}\limits^{A_{X,x}}\ \La_p^r(X,x;E)\ \mathop{\longrightarrow}\limits^{B_{X,x}}\  \La_{p+k}^r(X,x;E)
\end{equation}
where $x$ means that we take germs at $x$ of  $C^r$ sections of the corresponding bundles.

The proof of Theorem \ref{thm3} and its special case, Theorem \ref{thm1}, follows from the following lemmata. In all of them we assume that $X\subset M$ is a closed subset which, in the $C^\omega$ and $C^{hol}$ categories, is an analytic subset.

\begin{lemma}\label{lem-germs}
(a) In all three categories for an arbitrary point $x\in X$ we have $\Ker B_{X,x}=\im A_{X,x}$ if $p<\dpt I_{X,x}^r(\Om)$.

\noindent (b) The same holds in the $C^\infty$ category if $M$, $E$, $\o_1,\dots,\o_k$ are of class $C^\omega$, the subset $X\subset M$ is real analytic, and the depth inequality holds for $I_{X,x}^\omega(\Om)$.
\end{lemma}

\begin{proof}
Statement (a) follows directly from Theorem \ref{thmSaito} applied to the module $\MM=\Gamma^r(X,x;E)$ of germs at $x$ of $C^r$ sections $X\to E$ over the ring $C^r_x(X)$ if we take into account that $\La^r_p(X,x;E)\simeq \w^p\MM$.

Statement (b) follows from statement (a). Namely, the ring of smooth function germs, say at the origin in $\R^n$, is flat over the ring of real analytic function germs at the origin  (\cite{Ml}, Chapter VI, Corollary 1.12). This means that exactness of the complex \eqref{cABX-germs} for germs in the real analytic category, which holds by (a), implies its exactness in the smooth category.
\end{proof}

\begin{lemma}\label{lem-local-global}
In any of the categories $C^r$, $r\in\{\infty, \omega, hol\}$, assuming in the category $C^\omega$ that $X\subset M$ is a coherent analytic subset, the global complex \eqref{opABXc} is exact if and only if its germ version \eqref{cABX-germs} is exact for all $x\in X$.
\end{lemma}

\begin{proof}
Let us first note that the modules of sections appearing in the complex \eqref{opABXc} are well defined over open subsets of $X$. Thus sheaves of modules
\begin{equation}\label{3sheaves}
\SS^r(X,(\w^{p-1}E)^k),\ \ \SS^r(X,\w^p E),\ \  \SS^r(X,\w^{p+k}E)
\end{equation}
of $C^r$ section-germs of the corresponding bundles over $X$ are well defined (over the sheaves of rings of $C^r$ function-germs on $X$).

The lemma can be reduced to Proposition \ref{prop1}. We first show this in the case of smooth category. As mentioned in Section \ref{s-sheaves}, the sheaf of rings $\RR$ of smooth function germs on a second countable manifold $M$ has the partition of unity property and the property (P). Similarly, the sheaf of rings $\RR$ of function germs on $X$ defined by the sections in $C^\infty(X\cap U)=C^\infty(U)/C^\infty(U,X\cap U)$, for any open $U\subset M$, has analogous properties. This means that Proposition \ref{prop1} is applicable to the complexes \eqref{opABXc} and \eqref{cABX-germs} and it implies Lemma \ref{lem-local-global} in the case of smooth category.

In order to reduce the lemma to Proposition \ref{prop1} in the categories $C^\omega$ and $C^{hol}$ we will first prove that in these cases the above sheaves are coherent.  As in both cases the proof is analogous we fix our attention on one of them, the category $C^{hol}$. Denote by $\OO(M)$ and $\OO(X)$ the sheaves of germs of complex analytic functions on $M$ and $X$, respectively. Then $\OO(X)$ is the quotient sheaf $\OO(M)/\OO(M,X)$, with $\OO(M,X)$ the subsheaf of analytic functions vanishing on $X$. Since $X$ was assumed to be a closed analytic subset of the Stein manifold $M$, both sheaves $\OO(M)$ and $\OO(M,X)$ are coherent, by a theorem of Oka. Therefore, the sheaf $\RR(X)$ is coherent as a quotient of two coherent sheaves. Furthermore, the sheaf $\OO(M,\w^p E)$ of germs of holomorphic sections of the bundle $\w^pE$ is coherent since coherence is a local property and, for any contractible pseudoconvex open subset $U\subset M$,  $\OO(U,\w^p E)$ is isomorphic to a Cartesian product $(\OO(U))^k$ with $k$ a positive integer. The subsheaf $\OO(M,X;\w^p E)$ of sections vanishing on $X$ is also coherent and again $\OO(X,\w^p E)$ is coherent as the quotient of these two sheaves. In an analogous way one can show that the sheaves $\OO(X,(\w^{p-1}E)^k)$ and $\OO(X,\w^{p+k}E)$ are coherent.

Since $A_X$ and $B_X$ define morphisms of the sheaves $\OO(X,(\w^{p-1}E)^k)\to \OO(X,\w^p E)$ and $\OO(X,\w^p E)\to \OO(X,\w^{p+k}E)$, we additionally have that the kernel subsheaves $\Ker A_X$ and $\Ker B_X$ are coherent by Theorem 2, Section 13 in \cite{Se}.

The proof of coherence of the above sheaves in the real analytic category is analogous, where we additionally use the assumption that the closed subset $X\subset M$ is coherent, i.e. the sheaf of rings of $C^\omega$ functions vanishing on $X$ is coherent. (In the complex category it is always coherent.) Moreover, by the same theorem in \cite{Se} the kernel subsheaves $\Ker A_X$ and $\Ker B_X$ are also coherent.

From the fact that the sheaves in \eqref{3sheaves} are coherent it follows that they satisfy properties (A) and (B) in Section \ref{s-sheaves}.
Namely, in the holomorphic category this follows from Theorems A and B of H. Cartan which hold for Stein manifolds. In the real analytic category this can be deduced from real versions of Cartan's Theorems A and B proved in the case of $M=\R^n$ in \cite{C}, Theorem 3, and from the Grauert embedding theorem \cite{G} saying that every $C^\omega$ manifold, countable at infinity, can be embedded as a closed $C^\omega$ submanifold in $\R^n$, for suitable $n$ (for this version of Cartan's theorem see also  \cite{GMT}, III. 3.7).

Since in both real analytic and holomorphic categories the involved sheaves have the properties (A) and (B) required in Proposition \ref{prop1}, the exactness of the global complex \eqref{opABXc} is equivalent to the exactness of its germ version \eqref{cABX-germs} for any $x\in X$.
\end{proof}

\it Proof of Theorems \ref{thm1} and \ref{thm3}. \rm
It is enough to prove Theorem \ref{thm3} as Theorem \ref{thm1} is its special case. Consider the global complex \eqref{opABXc}. The assumptions of the theorem imply that statement (a) in Lemma \ref{lem-germs} can be applied for any positive $p$ such that $p<\dpt I^r_x(\Om)$ for all $x\in \Sing(\Om)$. It can also be applied for $x\notin \Sing(\Om)$ since in this case the ideal $I^r_x(\Om)$ coincides with the ring of all germs at $x$ and, consequently, $\dpt I^r_x(\Om)=\infty$. It follows from the lemma that the local (germ) version of this complex is exact for germs at any $x\in M$. Using Lemma \ref{lem-local-global} we deduce then that the global complex \eqref{opABXc} is exact in the considered categories.
\hfill$\Box$
\bigskip

\it Proof of Corollaries \ref{cor1} and \ref{cor2}. \rm
The proof is analogous to the above proof where instead of using statement (a) of Lemma \ref{lem-germs} one should use statement (b) of that lemma.
\hfill$\Box$
\bigskip

In order to prove Theorems \ref{thm2} and \ref{thm4} we will use the following lemmata. The first one, needed for proving the next lemma, says that the extension property for functions implies the extension property for sections of a vector bundle.

\begin{lemma}\label{lem-split0}
If $E$ is a smooth vector bundle over a smooth manifold $M$ and a closed subset $X\subset M$ has the extension property then there exists a linear continuous operator $L:C^\infty(X;E)\to C^\infty(M;E)$ such that $L(e)\vert_X=e$, i.e., the natural quotient map $Q:C^\infty(M;E)\to C^\infty(X;E)$ has a continuous linear right inverse.
\end{lemma}

\begin{proof}
Let $S:C^\infty(X)\to C^\infty(M)$ be an extension operator. For a trivial bundle $E$ we can choose its global basis and use the extension operator component-wise. In general, if $\dim M=n$ and $E$ has rank $m$ then there exist generic smooth global sections $e_1,\dots,e_{m+n}:M\to E$ such that at each point of $M$ they span the whole fiber of $E$. This follows from the transversality theorem (see e.g. \cite{H}, Ch. 3, Theorem 2.1). Namely, the set of $m\times (m+n)$ matrices having rank smaller then $m$ is a stratified set of codimension $n+1$ (a finite set of submanifolds of codimension $n+1$ and larger) in the space of all real matrices of this size. The same property can be associated to collections of linearly dependent $m+n$ vectors in the fibers of $E$. By the transversality theorem there is a generic sequence of smooth sections $(e_1,\dots,e_{m+n}):M\to E\oplus\cdots\oplus E$ ($(n+m)$-times) transversal at each point $x$ to these submanifolds. Since the codimensions of them are larger then the dimension of $M$, transversality means no intersection, i.e. at no point the sections $(e_1,\dots,e_{m+n})$ span a subspace smaller then the fiber $E_x$.

Having such sections we construct the operator $L$. Let $\{U_i\}_{i\ge 1}$ and $\{V_i\}_{i\ge 1}$ be locally finite open covers of $M$ such that $U_i\subset\!\subset V_i\subset\!\subset M$ and for each $i$ there is a fixed subset of $n$ sections, among sections $e_1,\dots,e_{m+n}$, which are linearly independent on $V_i$. Then, given a section $e:X\to E$, we can find smooth functions $f^i_1,\dots,f^i_{m+n}$ on $V_i\cap X$ such that $e\vert_{V_i\cap X}=\sum_j f^i_je_j$ on $V_i\cap X$. Let $(\varphi_i)$ be a partition of unity on $X$ subordinate to the cover $\{X\cap U_i\}$ and let $\psi_i:M\to [0,1]$ be smooth functions such that $\supp \psi_i\subset V_i$ and $\psi_i\vert_{U_i}=1$. We define an extension operator
\[
Le=\sum_{1\le j\le m+n}\sum_{i\ge 1}\psi_iS(\varphi_i f^i_j)e_j,
\]
where $\varphi_if^i_j$ are extended by zero from $X\cap V_i$ to $X$ and $S(\varphi_if^i_j)=0$ if $U_i\cap X=\emptyset$.
The sum is locally finite and on any $U_i\cap X$ it is equal $\sum_i\sum_j\varphi_if^i_je_j=\sum_i\varphi_ie$=e. Thus,
$L:C^\infty(X;E)\to C^\infty(M;E)$ is well defined, continuous in the $C^\infty$ topologies, and $(Le)\vert_X=e$, i.e., $L$ is a continuous right inverse to the natural quotient operator $Q_E:C^\infty(M;E)\to C^\infty(X;E)$.
\end{proof}

The following results are simplified versions of Lemmata 3.4 and 3.6 from \cite{DJ}. Recall that a subspace $Y_1$ of a Fr\'echet space $Y$ is called \emph{complemented} if it is closed and there exists another closed subspace $Y_2$ such that $Y=Y_1\oplus Y_2$.

Let $\ss$ denote the space of rapidly decreasing real or complex sequences
\[
\ss=\{\,a=(a_n)_{n\in\N}\,:\ \forall\, k\ \|a\|_k:=\sup_{n\in \N}\{n^k|a_n|<\infty \}\,\}.
\]
With the topology defined by the sequence of seminorms $\|\cdot\|_k$ this is a nuclear Fr\'echet space isomorphic, via the Fourier transform, to the space $C^\infty(S^1)$. Let $\ss^\infty$ denote the countable product of copies of the Fr\'echet space $\ss$, where the topology in $\ss^\infty$ is given by the sequence of topologies induced by projections onto finite subproducts. Note that $\ss$ is a complemented subspace of $\ss^\infty$.

\begin{lemma}\label{lem-split1}
If $M$ is a non-compact $C^\infty$ manifold then the spaces $C^\infty(M)$ and $\ss^\infty$ are isomorphic as Fr\'echet spaces.
If $X\subset M$ has the extension property then $C^\infty(M,X)$ and $C^\infty(X)$ are isomorphic to complemented subspaces of $C^\infty(M)$ and, thus, to complemented subspaces of $\ss^\infty$. Under the same assumptions the space of smooth sections $C^\infty(M;E)$ of a vector bundle $E$ is isomorphic to $\ss^\infty$ and the spaces $C^\infty(M,X;E)$, $C^\infty(X;E)$ are isomorphic to complemented subspaces of $\ss^\infty$. The same holds in the case of $M$ compact if the space $\ss^\infty$ is replaced with $\ss$.
\end{lemma}

\begin{proof}
For $M$ non-compact manifold the statements on $C^\infty(M)$ and $C^\infty(M;E)$ are consequences of Theorem 2.4 in \cite{DV}.
If $M$ is compact then, using the same proof as of Theorem 2.3 in \cite{V}, one can show that the Fr\'echet spaces $C^\infty(M)$ and $C^\infty(M;E)$ are isomorphic to the Fr\'echet space $\ss$.

If $X$ has the extension property then denoting the extension operator by $S:C^\infty(X)\to C^\infty(M)$ and the restriction (quotient) map by $Q:C^\infty(M)\to C^\infty(X)$ we see that the operator $P=SQ:C^\infty(M)\to C^\infty(M)$ is a continuous projection and the space $C^\infty(M)$ is the direct sum
\[
C^\infty(M)=\Ker(P)\oplus \Ker(I-P)=C^\infty(M,X)\oplus \im S,
\]
thus $C^\infty(M,X)$ and $C^\infty(X)\simeq\im S$ are complemented subspaces of the space $C^\infty(M)$ isomorphic to $\ss$, if $M$ is compact, and to $\ss^\infty$ if $M$ is not compact. The analogous property of the spaces of sections of $E$ follows by using Lemma \ref{lem-split0} and replacing the operator $S$ above by the operator $L$. One then obtains that $C^\infty(M;E)=C^\infty(M,X;E)\oplus \im L$ and thus $C^\infty(M,X;E)$ and $C^\infty(X;E)\simeq\im L$ are complemented subspaces of the space $C^\infty(M;E)$ isomorphic to $\ss$ or $\ss^\infty$.
\end{proof}

\begin{lemma}[\cite{DJ}, Lem. 3.6]\label{lem-split2}
Every exact complex of Fr\'echet spaces
\[
0\longrightarrow X_0\mathop{\longrightarrow}\limits^{T_0} X_1 \mathop{\longrightarrow}\limits^{T_1} X_2\mathop{\longrightarrow}\limits^{T_2} \ \cdots\ \mathop{\longrightarrow}\limits^{T_n} X_{n+1}
\]
splits (i.e. for $i=0,\dots,n-1$ the operators $T_i:X_i\to \im T_i$ have continuous linear right inverses) if $X_0,\dots,X_n$ are complemented subspaces of $\ss^\infty$.
\end{lemma}

\begin{remark}\rm
The Fr\'echet spaces $C^\infty(M)$, $C^\infty(M,E)$, $s$ and $\ss^\infty$ can be endowed with natural gradings and then all morphisms (isomorphisms) appearing in the above lemmata are graded morphisms (isomorphisms). See \cite{DV} and \cite{DJ} for details.
\end{remark}

\it Proof of Theorem \ref{thm2}. \rm
The complex \eqref{cABMs} can be completed to the complex
\begin{equation}\label{cABMc}
0\longrightarrow \Ker A \longrightarrow (\La^\infty_{p-1}(M;E))^k \mathop{\longrightarrow}\limits^{A} \La_p^\infty(M;E) \mathop{\longrightarrow}\limits^{B} \La_{p+k}^\infty(M;E),
\end{equation}
where the second arrow denotes the operator of inclusion. In order to prove the first statement note that the complex is exact at the first two spaces (by definitions of the arrows) and it is also exact at the third space by the assumption of the theorem. The space $(\La^\infty_{p-1}(M;E))^k$ of smooth sections of the vector bundle $(\w^{p-1}E)^k$ over $M$ is isomorphic to the Fr\'echet space $\ss^\infty$, or to $\ss$, by Lemma \ref{lem-split1}. By the same lemma the spaces $\La_p^\infty(M;E)$ and $\La_{p+k}^\infty(M;E)$ of smooth sections of the vector bundles $\w^pE$ and $\w^{p+k}E$ over $M$ are isomorphic to $\ss$ or $\ss^\infty$. We can thus use Lemma \ref{lem-split2} to deduce that the exact sequence \eqref{cABMc} splits, in particular \eqref{cABMs} splits. The second statement follows from the first and from Theorem \ref{thm1}.
\hfill $\Box$
\medskip

\it Proof of Theorem \ref{thm4}. \rm
The proof is analogous to the above proof with the spaces $(\La^\infty_{p-1}(M;E))^k$, $\La_p^\infty(M;E)$, $\La_{p+k}^\infty(M;E)$ replaced with $(\La^\infty_{p-1}(X;E))^k$, $\La_p^\infty(X;E)$, and $\La_{p+k}^\infty(X;E)$. In this case, however, in order to use Lemma \ref{lem-split1} we have to check that the subset $X\subset M$ has the extension property. This is assumed in the first statement of the theorem and the conclusion follows from Lemma \ref{lem-split2}. In the second statement the extension property follows from our assumptions that $M$ is a real analytic manifold and $X$ is a closed real analytic subset of $M$ and from Theorem 0.2.1 in \cite{BS}. Then the second statement follows from the first and Theorem \ref{thm3}.
\hfill $\Box$

\end{document}